\documentclass[11pt]{amsart}
\usepackage[OT2,T1]{fontenc}
\usepackage{tikz-cd}
\usepackage{amsmath,amsfonts,amsthm,amssymb,mathrsfs,
amsxtra,amscd,latexsym,color,paralist,bbm}
\usepackage[hmargin=2.5cm,vmargin=3.5cm]{geometry}
\usepackage[normalem]{ulem}
\usepackage{soul}
\usepackage{setspace}
\DeclareSymbolFont{cyrletters}{OT2}{wncyr}{m}{n}
\DeclareMathSymbol{\Sha}{\mathalpha}{cyrletters}{"58}

 \newtheorem{thm}{Theorem}[section]
 \newtheorem*{thm*}{Theorem}

 \newtheorem{prop}[thm]{Proposition}
 
 \newtheorem{dfn}[thm]{Definition}

 \theoremstyle{definition}
 \theoremstyle{remark}
 \numberwithin{equation}{section}

\newcommand{\sm}{\left(\begin{smallmatrix}}
\newcommand{\esm}{\end{smallmatrix}\right)}
\newcommand{\bpm}{\begin{pmatrix}}
\newcommand{\ebpm}{\end{pmatrix}}
\newcommand{\mat}{\left(\begin{matrix}}
\newcommand{\emat}{\end{matrix}\right)}

\def\ZZ{\mathbb{Z}}

\def\GL{\mathrm{GL}}

\def\SL{\mathrm{SL}}

\makeatletter
\@namedef{subjclassname@2020}{\textup{2020} Mathematics Subject Classification}
\makeatother

\begin{document}

\title{Rankin-Cohen bracket for vector-valued modular forms}

\author{Youngmin Lee}
\address{School of Mathematics, Korea Institute for Advanced Study, 85 Hoegiro, Dongdaemun-gu,
Seoul 02455, Republic of Korea}
\email{youngminlee@kias.re.kr}
 \author{Subong Lim}
   \address{Department of Mathematics Education, Sungkyunkwan University, Jongno-gu, Seoul 110-745, Republic of Korea}
   \email{subong@skku.edu}
   \author{Wissam Raji}
   \address{Department of Mathematics, American University of Beirut (AUB) and the Number Theory Research Unit at the Center for Advanced Mathematical Sciences (CAMS) at AUB, Beirut, Lebanon}
   \email{wr07@aub.edu.lb}




\begin{abstract}
In this paper, we explore the relationship between Rankin-Cohen brackets for vector-valued modular forms and Petersson’s inner products, deriving an explicit description of the adjoint map for this bracket operator. The study extends to the cases of Jacobi forms and skew-holomorphic Jacobi forms, establishing connections between their respective Rankin-Cohen brackets and those defined for vector-valued modular forms through an isomorphism. Adjoint maps for these extended bracket operators are also examined.
\end{abstract}
\maketitle
\section{Introduction} Vector-valued modular forms are particularly significant in number theory and representation theory. These forms are important generalizations of elliptic modular forms that arise naturally in the theory of Jacobi forms, Siegel modular forms, and Moonshine.  Vector-valued modular forms have been used as an important tool in tackling classical problems in the theory of modular forms.  For example, Selberg used these forms to give estimates for the Fourier coefficients of the classical modular forms \cite{S}. Borcherds in \cite{B1} and \cite{B2} used vector-valued modular forms associated with Weil representations to describe the Fourier expansion of various theta liftings.  Applications of vector-valued modular forms are found in several areas. For instance, they play a crucial role in studying certain differential equations and are connected to vertex operator algebras and conformal field theory in physics. Other applications concerning vector-valued modular forms of half-integral weight seem to provide a simple solution to the Riemann-Hilbert problem for representations of the modular group \cite{BG}.

\par On the other hand, Rankin-Cohen brackets hold significant importance in the theory of classical modular forms, providing a powerful method for constructing new modular forms and revealing deep structural and arithmetic properties. Beyond construction, these brackets endow the graded ring of modular forms with a rich algebraic structure, often referred to as a Rankin-Cohen algebra, and are instrumental in understanding differential equations satisfied by modular forms, such as Ramanujan's differential equations for Eisenstein series. They also play a crucial role in number theory, notably through connections established by Zagier to the periods of modular forms and the special values of L-functions. They have been used to prove various arithmetic identities involving Fourier coefficients and divisor sums. Generalizations of Rankin-Cohen brackets have been explored for Calabi-Yau quasi-modular forms, indicating potential connections to geometry \cite{N}. In this paper, we generalize Rankin-Cohen brackets for the case of vector-valued modular forms and explore their relationship with Jacobi forms, demonstrating an isomorphism between spaces of Jacobi forms and certain spaces of vector-valued modular forms. This connection allows for the transfer of structures between them. Rankin-Cohen brackets are differential operators that construct new modular forms from existing ones. In the context of vector-valued modular forms, these brackets are particularly important because they provide a way to generate new vector-valued modular forms from existing ones, often with higher weights. The paper defines the $\nu$-th Rankin-Cohen bracket $[f_1, f_2]_\nu$ for two vector-valued modular forms $f_1$ and $f_2$ and shows that if $f_1 \in M_{k_1, \chi_1, \rho_1}$ and $f_2 \in M_{k_2, \chi_2, \rho_2}$, then their bracket $[f_1, f_2]_\nu$ is a vector-valued cusp form in $S_{k_1+k_2+2\nu, \chi_1\chi_2, \rho_1 \otimes \rho_2}$. 

In this paper, we focus on the Petersson pairing between a cusp form and the Rankin–Cohen bracket of a modular form with either an Eisenstein series or a Poincaré series.
It was initiated by Zagier \cite{ZA} and has been studied in various directions. 
For instance, Choie, Kohnen, and Zhang \cite{CKZ} extended Zagier’s results from Eisenstein series to Poincaré series and considered more general multiplier systems.
From the perspective of vector-valued modular forms, Ni and Xue \cite{NX} obtained results for the case of integral weights and Eisenstein series.
In Theorem \ref{thm 3}, we provide the explicit formula of $\langle f,[g,\mathbb{P}]_{\nu} \rangle$ in terms of the Fourier coefficients of $f$ and $g$, where $f$ is a vector-valued cusp form, $g$ is a vector-valued modular form, and $\mathbb{P}$ denotes a Poincaré series.
This result generalizes previous results to vector-valued modular forms with real weights and general multiplier systems involving Poincaré series.

The remainder of this paper is organized as follows. 
In Section \ref{s : Pre}, we review the notions and properties of vector-valued modular forms including the Rankin-Cohen bracket and a Poincaré series.
In Section \ref{s : Thm}, we prove our main theorem on the formula for the Petersson pairing between a vector-valued cusp form and the Rankin-Cohen bracket of a vector-valued modular form with a Poincaré series. 
Moreover, we introduce a result on the adjoint map of the Rankin-Cohen bracket with respect to the Petersson inner product. 
In Sections \ref{s : Jacobi forms} and \ref{s : skew Jacobi}, we recall the concept of Rankin-Cohen brackets involving Jacobi forms and skew-holomorphic Jacobi forms each paired with modular forms, and demonstrate how these constructions relate to the vector-valued setting via an isomorphism. 
Furthermore, building on the results from Section \ref{s : Thm}, we present corresponding results for Jacobi forms and skew-holomorphic Jacobi forms in Sections \ref{s : Jacobi forms} and \ref{s : skew Jacobi}, respectively.

\section{Preliminaries}\label{s : Pre}
Let $d$ be a positive integer and $\rho : \mathrm{SL}_{2}(\mathbb{Z})\to \GL_{d}(\mathbb{C})$ be a representation.
Let $\alpha(\rho)$ be a non-negative real number depending on $\rho$ (for details, see \cite{KM2}).
Note that if $\rho$ is a unitary representation, then $\alpha(\rho)=0$.
Let $\chi$ be a unitary multiplier system of weight $k$ on $\mathrm{SL}_{2}(\mathbb{Z})$. 
In this paper, we assume that $\rho\left(T\right)$ is a diagonal unitary matrix, where $T:=\sm 1 & 1 \\
0 & 1 \esm$.
Let $\{\mathbf{e}_{1},\dots,\mathbf{e}_{d}\}$ be the standard basis of $\mathbb{C}^{d}$.
Let $\mathbb{H}$ be the upper half of the complex plane. 
For a vector-valued modular form $f=\sum_{j=1}^{d} f_{j} \mathbf{e}_{j} : \mathbb{H} \to \mathbb{C}^{d}$ and $\gamma\in \mathrm{SL}_{2}(\mathbb{Z})$, the slash operator $|_{k,\chi,\rho}\gamma$ is defined by 
\[ \left(f|_{k,\chi,\rho} \gamma\right)(\tau):=\chi^{-1}(\gamma)(c\tau+d)^{-k}\rho^{-1}(\gamma)f(\gamma \tau), \]
where $\gamma=\sm a & b\\
c & d\esm\in \mathrm{SL}_{2}(\mathbb{Z})$.

\begin{dfn}
    A holomorphic function $f:=\sum_{j=1}^{d} f_{j}\mathbf{e}_{j} : \mathbb{H}\to \mathbb{C}^{d}$ is called a vector-valued modular form of weight $k$ and multiplier system $\chi$ with respect to $\rho$ on $\mathrm{SL}_{2}(\mathbb{Z})$ if
    \begin{enumerate}
        \item $f|_{k,\chi,\rho} \gamma = f$ for all $\gamma\in \mathrm{SL}_{2}(\mathbb{Z})$, and
        \item for each $j\in \{1,\dots, d\}$, $f_{j}$ has the Fourier expansion of the form 
        \[ f_{j}(\tau)=\sum_{n=0}^{\infty} a_{j}(n+\kappa_j) e^{2\pi i (n+\kappa_j)\tau}. \]
        Here, $\kappa_{j}$ denote a real number in $[0,1)$ such that the $(j,j)$-th entry of  $\chi\rho\left(T\right)$ is $e^{2\pi i \kappa_{j}}$.
    \end{enumerate}
    If a vector-valued modular form $f$ satisfies $a_{j}(0)=0$ for all $j\in \{1,\dots,d\}$, then we call that $f$ is a vector-valued cusp form. 
    Let $M_{k,\chi,\rho}$ (resp. $S_{k,\chi,\rho}$) be the space of vector-valued modular forms (resp. cusp forms) of weight $k$ and multiplier system $\chi$ with respect to $\rho$ on $\mathrm{SL}_{2}(\mathbb{Z})$. 
\end{dfn}

  For a holomorphic function $f:=\sum_{j=1}^{d}f_j \mathbf{e}_{j}:\mathbb{H}\to \mathbb{C}^{d}$ and for an integer $s$, we define
    \[ f^{(s)}(\tau):=\sum_{j=1}^{d} D^s_{\tau} (f_j) \mathbf{e}_{j}, \]
    where $D_{\tau}:=\frac{1}{2\pi i}\cdot\frac{d}{d\tau}$.

\begin{dfn}\label{def 1}\cite[Definition 2.8]{NX}
    For each $i\in \{1,2\}$, let $d_i$ be a positive integer.
    Assume that $\rho_i : \mathrm{SL}_{2}(\mathbb{Z})\to \GL_{d_i}(\mathbb{C})$ is a representation such that $\rho_i(T)$ is a diagonal unitary matrix. 
    Let $k_i$ be a real number and $\chi_i$ be a multiplier system of weight $k_i$ on $\mathrm{SL}_{2}(\mathbb{Z})$.
    For two vector-valued modular forms $f_1\in M_{k_1,\chi_1,\rho_1}$, $f_2\in M_{k_2,\chi_2,\rho_2}$ and for a positive integer $\nu$, we define $\nu$-th Rankin-Cohen bracket $[f_1,f_2]_{\nu}$ of $f_1$ and $f_2$ by
    \begin{equation*}
        [f_1,f_2]_{\nu}:=\sum_{i=0}^{\nu}(-1)^{\nu-i}\binom{\nu}{i}\frac{\Gamma(k_1+\nu)\Gamma(k_2+\nu)}{\Gamma(k_1+i)\Gamma(k_2+\nu-i)}f_1^{(i)}(\tau) \otimes f_2^{(\nu-i)}(\tau).
    \end{equation*}
\end{dfn}

\begin{thm}\label{thm 1}\cite[Theorem 1.3]{Z}
    With the notation in Definition \ref{def 1}, we have 
    \[ [f_1,f_2]_{\nu}\in S_{k_1+k_2+2\nu,\chi_1\chi_2, \rho_1\otimes \rho_2}. \]
\end{thm}

We review the definition of Poincar\'e series for vector-valued modular forms in \cite[Section 3]{KM}.
\begin{dfn}\label{def 2}
Let $d$ be a positive integer and $\rho : \mathrm{SL}_{2}(\mathbb{Z})\to \mathrm{GL}_{d}(\mathbb{C})$ be a representation.
    Let $k$ be a real number with $k>2+2\alpha(\rho)$ and $\chi$ be a multiplier system of weight $k$ on $\mathrm{SL}_{2}(\mathbb{Z})$.
    Let $s$ be a non-negative integer and $u$ be a positive integer with $1\leq u\leq d$.
    A vector-valued Poincar\'{e} series $\mathbb{P}(\tau;\rho,k,\chi,s,u)$ is defined by 
    \[ \mathbb{P}(\tau;\rho,k,\chi,s,u):=\frac{1}{2} \sum_{M\in \mathrm{SL}_{2}(\mathbb{Z})/\langle T \rangle } \chi^{-1}(M)(c\tau+d)^{-k}\mathrm{exp}(2\pi i(s+\kappa_u)M\tau)\rho^{-1}(M)\mathbf{e}_{u}. \]
    Here, $\exp(z):=e^{z}$.
\end{dfn}
Note that $\mathbb{P}(\tau;\rho,k,\chi,s,u)\in M_{k,\chi,\rho}$ if $s$ is a non-negative integer and $k>2+2\alpha(\rho)$.
Now, we recall the definition of the Petersson pairing for vector-valued modular forms in \cite[Section 5]{KM}.
For a representation $\rho:\mathrm{SL}_{2}(\mathbb{Z})\to \GL_d(\mathbb{C})$, let $\rho^{\vee} : \mathrm{SL}_{2}(\mathbb{Z})\to \GL_d(\mathbb{C})$ be the dual representation of $\rho$ defined by  
\[ \rho^{\vee}(g):=\overline{\rho(g^{-1})^{T}}. \]

\begin{dfn}
   Assume that $f=\sum_{j=1}^{d} f_{j} \mathbf{e_j}\in M_{k,\chi,\rho}$ and $g=\sum_{j=1}^{d} g_{j}\mathbf{e_j}\in S_{k,\chi,\rho^{\vee}}$. 
   The Petersson pairing $\langle f , g \rangle$ of $f$ and $g$ is defined by
   \[ \langle f, g \rangle := \int_{\mathrm{SL}_{2}(\mathbb{Z})\backslash \mathbb{H}} \sum_{j=1}^{d} f_{j}(\tau)\overline{g_j(\tau)} y^{k}\frac{dxdy}{y^2},  \]
   where $\tau=x+iy$.
\end{dfn}

\begin{thm}\cite[Theorem 5.3]{KM}\label{thm 2}
    Let $\rho,k,\chi,s,u$ be as in Definition \ref{def 2}.
    Let $g(\tau)=\sum_{j=1}^{d}\sum_{n=0}^{\infty}b_j(n+\kappa_j)e^{2\pi i (n+\kappa_j)\tau} \mathbf{e}_{j}\in S_{k,\chi,\rho^{\vee}}$.
    Then, 
    \[ \langle \mathbb{P}(\tau;\rho,k,\chi,s,u) , g \rangle = \begin{cases}
        \overline{b_{u}(s+\kappa_r)}\frac{\Gamma(k-1)}{(4\pi(s+\kappa_u))^{k-1}} \quad &\text{if }s+\kappa_u>0,\\
        0 &\text{if }s+\kappa_u=0.
    \end{cases}  \]
\end{thm}

\section{Theorems}\label{s : Thm}

In this section, we aim to prove the main theorem of this paper, which gives an explicit formula for the Petersson pairing between a vector-valued cusp form and the Rankin–Cohen bracket of a vector-valued modular form and a Poincaré series.
Based on this result, we also investigate the adjoint maps of these bracket operators with respect to the Petersson inner product.
The main result of this paper is stated in the following theorem.

\begin{thm}\label{thm 3}
    For each $i\in \{1,2\}$, let $d_i$, $k_i$, $\chi_i$ and $\rho_i$ for $i\in \{1,2\}$ be as in Definition \ref{def 1}.
    Assume that $k_i>2+2\alpha(\rho_i)$.
    For each $j\in \{1,\dots,d_i\}$, let $\kappa_{i,j}$ be a real number in $[0,1)$ such that the $(j,j)$-th entry of $\chi_i\rho_i\left(T\right)=e^{2\pi i \kappa_{i,j}}$.
    Let $\nu$ be a positive integer. 
    Let $s$ be a non-negative integer and $r$ be a positive integer with $1\leq r\leq d_2$.
    For convenience, let $\mathbb{P}:=\mathbb{P}(\tau;\rho_2,k_2,\chi_2,s,r)$. 
    Assume that 
    \[f(\tau)=\sum_{j=1}^{d_1}\sum_{l=1}^{d_2} a_{j,l}(n+\kappa_{1,j}+\kappa_{2,l})e^{2 \pi i(n+\kappa_{1,j}+\kappa_{2,l})\tau}\mathbf{e_j}\otimes \mathbf{e_l'}\in S_{k_1+k_2+2\nu,\chi_1\chi_2,(\rho_1\otimes \rho_2)^{\vee}}\] 
    and 
    \[g(\tau)=\sum_{j=1}^{d_1}b_j(n+\kappa_{1,j}) e^{2\pi i (n+\kappa_{1,j})\tau} \mathbf{e_j}\in M_{k_1,\chi_1,\rho_1}.\]
    Then, we have 
    \begin{equation*}
        \begin{aligned}
            \langle f, [g,\mathbb{P}]_{\nu} \rangle &= \frac{\Gamma(k_1+k_2+2\nu-1)}{(4\pi)^{k_1+k_2+2\nu-1}}\sum_{j=1}^{d_1}\sum_{u=0}^{\nu} (-(s+\kappa_{2,r}))^{u}\binom{\nu}{u}\\
            &\times\frac{\Gamma(k_1+\nu)\Gamma(k_2+\nu)}{\Gamma(k_1+\nu-u)\Gamma(k_2+u)}\sum_{n\in \mathbb{Z}} (n+\kappa_{1,j})^{\nu-u}\frac{a_{j,r}(s+n+\kappa_{1,j}+\kappa_{2,r})\overline{b_j(n+\kappa_{1,j})}}{(s+n+\kappa_{1,j}+\kappa_{2,r})^{k_1+k_2+2\nu-1}}.
        \end{aligned}
    \end{equation*}
\end{thm}
\begin{proof}
    We follow the ideas of the proofs in \cite[Lemma 4.3]{NX} and \cite[Proposition 2.2]{CKZ}.
    By the definitions of Rankin-Cohen bracket and the Poincar\'e series for vector-valued modular forms, we have
    \begin{equation*}
        \begin{aligned}
            [g,\mathbb{P}]_{\nu}&=\sum_{i=0}^{\nu}(-1)^{\nu-i} \binom{\nu}{i} \frac{\Gamma(k_1+\nu)\Gamma(k_2+\nu)}{\Gamma(k_1+i)\Gamma(k_2+\nu-i)}g^{(i)}(\tau)\otimes \mathbb{P}^{(\nu-i)}(\tau)\\
            &=\sum_{i=0}^{\nu}(-1)^{\nu-i}\binom{\nu}{i}\frac{\Gamma(k_1+\nu)\Gamma(k_2+\nu)}{\Gamma(k_1+i)\Gamma(k_2+\nu-i)}g^{(i)}(\tau)\\
            &\otimes \left(\frac{1}{2} \sum_{M\in \mathrm{SL}_{2}(\mathbb{Z})/\langle T \rangle } \chi_2^{-1}(M)(c\tau+d)^{-k_2}\mathrm{exp}(2\pi i(s+\kappa_{2,r})M\tau)\rho_2^{-1}(M)\mathbf{e}_{r}' \right)^{(\nu-i)}.
        \end{aligned}
    \end{equation*}
    By induction, we obtain that for $M=\sm a & b\\
    c & d \esm\in \mathrm{SL}_{2}(\mathbb{Z})$, 
    \begin{equation*}
        \begin{aligned}
            &\left((c\tau+d)^{-k_2}\exp(2\pi i (s+\kappa_{2,r})M\tau)\right)^{(\nu-i)}\\
            &=\sum_{m=0}^{\nu-i}\binom{\nu-i}{m}\frac{\Gamma(k_2+\nu-i)}{\Gamma(k_2+m)}(-c)^{\nu-i-m}(s+\kappa_{2,r})^{m}(c\tau+d)^{-k_2-(\nu-i-m)}\exp(2\pi i (s+\kappa_{2,r})M\tau).
        \end{aligned}
    \end{equation*}
    Thus, we obtain
    \begin{equation*}
        \begin{aligned}
            [g,\mathbb{P}]_{\nu}&=\sum_{i=0}^{\nu}(-1)^{\nu-i}\binom{\nu}{i}\frac{\Gamma(k_1+\nu)\Gamma(k_2+\nu)}{\Gamma(k_1+i)\Gamma(k_2+\nu-i)}g^{(i)}(\tau)\\
            &\otimes \bigg(\frac{1}{2} \sum_{M\in \mathrm{SL}_{2}(\mathbb{Z})/\langle T \rangle } \chi_2^{-1}(M) \sum_{m=0}^{\nu-i} \binom{\nu-i}{m}\frac{\Gamma(k_2+\nu-i)}{\Gamma(k_2+m)}(-c)^{\nu-i-m}(s+\kappa_{2,r})^{m}\\
            &\times (c\tau+d)^{-k_2-(\nu-i-m)}\exp(2\pi i (s+\kappa_{2,r})M\tau)\rho_2^{-1}(M)\mathbf{e_r'}.  \bigg).
        \end{aligned}
    \end{equation*}
    Since $\binom{\nu}{i}\binom{\nu-i}{m}=\binom{\nu}{m}\binom{\nu-m}{i}$, the expression is symmetric in $i$ and $m$, and hence we get
    \begin{equation*}
        \begin{aligned}
            [g,\mathbb{P}]_{\nu}&=\frac{1}{2}\sum_{M\in \mathrm{SL}_{2}(\mathbb{Z})/\langle T \rangle } (\chi_1\chi_2(M))^{-1} \sum_{i=0}^{\nu}  (s+\kappa_{2,r})^{i}(-1)^{i}\binom{\nu}{i} \frac{\Gamma(k_1+\nu)\Gamma(k_2+\nu)}{\Gamma(k_2+i)\Gamma(k_1+\nu-i)}\\
            &\times (c\tau+d)^{-k_1-k_2-2\nu}\exp(2\pi i (s+\kappa_{2,r})M\tau)\\
            &\times\sum_{m=0}^{\nu-i} \chi_1(M) \binom{\nu-i}{m} \frac{\Gamma(k_1+\nu-i)}{\Gamma(k_1+m)}c^{\nu-i-m}(c\tau+d)^{k_1+\nu-i-m}(g^{(m)}(\tau)\otimes \rho_2^{-1}(M) \mathbf{e_r'}).
        \end{aligned}
    \end{equation*}
    By \cite[(76)]{ZA}, we have 
    \begin{equation*}
        \begin{aligned}
            \sum_{m=0}^{\nu-i}\chi_1(M) \binom{\nu-i}{m} \frac{\Gamma(k_1+\nu-i)}{\Gamma(k_1+m)}c^{\nu-i-m}(c\tau+d)^{k_1+\nu-i-m}g^{(m)}(\tau) = \rho_1^{-1}(M)g^{(\nu-i)}(M\tau).
        \end{aligned}
    \end{equation*}
    Thus, we get 
    \begin{equation*}
        \begin{aligned}
            [g,\mathbb{P}]_{\nu}&=\frac{1}{2}\sum_{M\in \mathrm{SL}_{2}(\mathbb{Z})/\langle T \rangle } (\chi_1\chi_2(M))^{-1} \sum_{i=0}^{\nu}  (s+\kappa_{2,r})^{i}(-1)^{i}\binom{\nu}{i} \frac{\Gamma(k_1+\nu)\Gamma(k_2+\nu)}{\Gamma(k_2+i)\Gamma(k+1+\nu-i)}\\
            &\times (c\tau+d)^{-k_1-k_2-2\nu}\exp(2\pi i (s+\kappa_{2,r})M\tau)\left(\rho_1^{-1}(M)g^{(\nu-i)}(M\tau)\otimes \rho_2^{-1}(M) \mathbf{e_r'} \right).
        \end{aligned}
    \end{equation*}
    For convenience, let $\chi:=\chi_1\chi_2$ and $\rho:=\rho_1\otimes \rho_2$. 
    Note that 
    \begin{equation*}
        \begin{aligned}
            &\left(\exp(2\pi i (s+\kappa_{2,r})\tau)g^{(\nu-i)}(\tau)\otimes \mathbf{e'_r} \right)|_{k_1+k_2+2\nu,\chi,\rho}M\\
            &=\left(\chi_1\chi_2(M)\right)^{-1}(c\tau+d)^{-k_1-k_2-2\nu}\exp\left(2\pi i (s+\kappa_{2,r})M\tau\right)\rho_1^{-1}(M)g^{(\nu-i)}(M\tau)\otimes \rho_2^{-1}(M)\mathbf{e'_r}.
        \end{aligned}
    \end{equation*}
    Hence, we get 
    \begin{equation*}
        \begin{aligned}
            [g,\mathbb{P}]_{\nu}&=\frac{1}{2}\sum_{M\in \mathrm{SL}_{2}(\mathbb{Z})/\langle T \rangle }\sum_{i=0}^{\nu}(s+\kappa_{2,r})^{i}(-1)^{i}\binom{\nu}{i}\frac{\Gamma(k_1+\nu)\Gamma(k_2+\nu)}{\Gamma(k_2+i)\Gamma(k_1+\nu-i)}\\
            &\times \left(\exp(2\pi i (s+\kappa_{2,r})\tau)g^{(\nu-i)}(\tau) \otimes \mathbf{e_r'}\right)|_{k_1+k_2+2\nu,\chi,\rho}M\\
            &=\frac{1}{2}\sum_{M\in \mathrm{SL}_{2}(\mathbb{Z})/\langle T \rangle}\sum_{i=0}^{\nu}\bigg( (s+\kappa_{2,r})^{i}(-1)^{i}\binom{\nu}{i}\frac{\Gamma(k_1+\nu)\Gamma(k_2+\nu)}{\Gamma(k_2+i)\Gamma(k_1+\nu-i)}\exp(2\pi i (s+\kappa_{2,r})\tau)\\
            &g^{(\nu-i)}(\tau) \otimes \mathbf{e_r'}\bigg)|_{k_1+k_2+2\nu,\chi,\rho}M.
        \end{aligned}
    \end{equation*}

    Since $g(\tau)=\sum_{j=1}^{d} b_j(n+\kappa_{1,j})e^{2\pi i (n+\kappa_{1,j})\tau}\mathbf{e_j}$, we have 
    \begin{equation*}
        \begin{aligned}
            g^{(\nu-i)}(\tau)=\sum_{j=1}^{d_1}\sum_{n=0}^{\infty} (2\pi i (n+\kappa_j))^{\nu-i} b_j(n+\kappa_{1,j}) e^{2\pi i (n+\kappa_{1,j})\tau}\mathbf{e_j}.
        \end{aligned}
    \end{equation*}
    Thus, we obtain
    \begin{equation}\label{eq 1}
        \begin{aligned}
            [g,\mathbb{P}]_{\nu}&=\frac{1}{2}\sum_{M\in \mathrm{SL}_{2}(\mathbb{Z})/\langle T \rangle}\sum_{i=0}^{\nu}\bigg( (s+\kappa_{2,r})^{i}(-1)^{i}\binom{\nu}{i}\frac{\Gamma(k_1+\nu)\Gamma(k_2+\nu)}{\Gamma(k_2+i)\Gamma(k_1+\nu-i)}\exp(2\pi i (s+\kappa_{2,r})\tau)\\
            &\times\sum_{j=1}^{d_1}\sum_{n=0}^{\infty}  (n+\kappa_{1,j})^{\nu-i} b_j(n+\kappa_{1,j}) \exp(2\pi i (n+\kappa_{1,j})\tau)\mathbf{e_j}
            \otimes \mathbf{e_r'}\bigg)|_{k_1+k_2+2\nu,\chi,\rho}M\\
            &=\frac{1}{2}\sum_{i=0}^{\nu}\sum_{j=1}^{d_1}\sum_{n=0}^{\infty}(s+\kappa_{2,r})^{i}(-1)^{i}\binom{\nu}{i}\frac{\Gamma(k_1+\nu)\Gamma(k_2+\nu)}{\Gamma(k_2+i)\Gamma(k_1+\nu-i)} (n+\kappa_{1,j})^{\nu-i} b_j(n+\kappa_{1,j})\\
            &\times \sum_{M\in \mathrm{SL}_{2}(\mathbb{Z})/\langle T \rangle} \left(\exp(2\pi i (s+n+\kappa_{1,j}+\kappa_{2,r})\tau) \mathbf{e_j}\otimes \mathbf{e_r'}\right)|_{k_1+k_2+2\nu,\chi,\rho}M.
        \end{aligned}
    \end{equation}
    
    For a non-negative integer $n$, let
    \[ \mathbb{P}_{s+n}(\tau):=\frac{1}{2}\sum_{M\in \mathrm{SL}_{2}(\mathbb{Z})/\langle T \rangle} \left(\exp(2\pi i (s+n+\kappa_j+\kappa'_r)\tau) \mathbf{e_j}\otimes \mathbf{e_r'}\right)|_{k_1+k_2+2\nu,\chi,\rho}M. \]
    Note that $\mathbb{P}_{s+n}(\tau)$ is a vector-valued Poincar\'e series of weight $k_1+k_2+2\nu$. 
    Combining \eqref{eq 1} and Theorem \ref{thm 2}, we complete the proof of Theorem \ref{thm 3}.
 \end{proof}

Assume that $g\in M_{k_1,\chi_1,\rho_1}$ and $\nu$ is a positive integer.
For $f\in S_{k_2,\chi_2,\rho_2}$, we define 
\[ T_{g,\nu}(f):=[f,g]_{\nu}. \]
By Theorem \ref{thm 1}, $T_{g,\nu}$ is a linear map from $S_{k_2,\chi_2,\rho_2}$ to $S_{k_1+k_2+2\nu,\chi_1\chi_2,\rho_1\otimes \rho_2}$. 
Since $\rho_1$ and $\rho_2$ are unitary representations, there is an adjoint map $T_{g,\nu}^{*}: S_{k_1+k_2+2\nu,\chi_1\chi_2,\rho_1\otimes\rho_2} \to S_{k_2,\chi_2,\rho_2}$ such that 
\[ \langle T_{g,\nu}(f),h \rangle = \langle f, T_{g,v}^{*}(h) \rangle  \]
for all $f\in S_{k_2,\chi_2,\rho_2}$ and $h\in S_{k_1+k_2+2\nu,\chi_1\chi_2,\rho_1\otimes\rho_2}$.
By using Theorem \ref{thm 3}, we describe the adjoint map in the following theorem.

\begin{thm}\label{thm 4}
    For each $i\in \{1,2\}$ and $j\in \{1,\dots,d_i\}$, let $d_i$, $k_i$, $\chi_i$ and $\rho_i$ for $i\in \{1,2\}$ be as in Definition \ref{def 1}.
     Assume that $k_i>2+2\alpha(\rho_i)$.
    For each $j\in \{1,\dots,d_i\}$, let $\kappa_{i,j}$ be as in Theorem \ref{thm 3}.
    Let $g(\tau)=\sum_{j=1}^{d_1}\sum_{n=0}^{\infty} b_j(n+\kappa_{1,j})e^{2\pi i (n+\kappa_{1,j})\tau}\mathbf{e_j}\in M_{k_1,\chi_1,\rho_1}$ and $\nu$ be a positive integer. 
    Let $T_{g,\nu} : S_{k_2,\chi_2,\rho_2}\to S_{k_1+k_2+2\nu,\chi_1\chi_2,\rho_1\otimes \rho_2}$ be the linear map defined by $T_{g,\nu}(f):=[f,g]_{\nu}$. 
    If 
    \[h(\tau)=\sum_{j=1}^{d_1}\sum_{l=1}^{d_2} \sum_{n\in \mathbb{Z}} a_{j,l}(n+\kappa_{1,j}+\kappa_{2,l})e^{2\pi i (n+\kappa_{1,j}+\kappa_{2,l})\tau}\mathbf{e_j}\otimes \mathbf{e_l'}\in S_{k_1+k_2+2\nu,\chi_1\chi_2,\rho_1\otimes\rho_2},\] 
    then 
    \[ T_{g,v}^{*}(h)(\tau)=\sum_{l=1}^{d_2}\sum_{n\in \mathbb{Z}}c_l(n+\kappa_{2,l})e^{2\pi i (n+\kappa_{2,l})\tau}\mathbf{e_l'}\in S_{k_2,\chi_2,\rho_2},\]
    where
    \begin{equation*}
        \begin{aligned}
            c_{l}(n+\kappa_{2,l})&=(-1)^{\nu}(4\pi(n+\kappa_{2,l}))^{k_2-1}\frac{\Gamma(k_1+k_2+2\nu-1)}{(4\pi)^{k_1+k_2+2\nu-1}}\sum_{j=1}^{d_1}\sum_{u=0}^{\nu} (-(n+\kappa_{2,l}))^{u}\binom{\nu}{u}\\
            &\times\frac{\Gamma(k_1+\nu)\Gamma(k_2+\nu)}{\Gamma(k_1+\nu-u)\Gamma(k_2+u)\Gamma(k_2-1)}\sum_{t\in \mathbb{Z}} (t+\kappa_{1,j})^{\nu-u}\frac{a_{j,l}(n+t+\kappa_{1,j}+\kappa_{2,l})\overline{b_j(t+\kappa_{1,j})}}{(n+t+\kappa_{1,j}+\kappa_{2,l})^{k_1+k_2+2\nu-1}}.
        \end{aligned}
    \end{equation*}
\end{thm}
\begin{proof}
By Theorems \ref{thm 2} and \ref{thm 3}, we have
\begin{equation*}
\begin{aligned}
    &c_l(n+\kappa_{2,l})\frac{\Gamma(k_2-1)}{(4\pi(n+\kappa_{2,l}))^{k_2-1}}\\
    &=\langle T_{g,\nu}^{*}(h),\mathbb{P}(\tau;k_2,\rho_2,\chi_2,n,l) \rangle\\
    &=\langle h, T_{g,\nu}(\mathbb{P}(\tau;k_2,\rho_2,\chi_2,n,l)) \rangle\\
    &= \langle h, [\mathbb{P}(\tau;k_2,\rho_2,\chi_2,n,l),g]_{\nu} \rangle\\
    &= (-1)^{\nu}\langle h, [g,\mathbb{P}(\tau;k_2,\rho_2,\chi_2,n,l)]_{\nu} \rangle\\
    &= (-1)^{\nu}\frac{\Gamma(k_1+k_2+2\nu-1)}{(4\pi)^{k_1+k_2+2\nu-1}}\sum_{j=1}^{d_1}\sum_{u=0}^{\nu} (-(n+\kappa_{2,l}))^{u}\binom{\nu}{u}\\
            &\times\frac{\Gamma(k_1+\nu)\Gamma(k_2+\nu)}{\Gamma(k_1+\nu-u)\Gamma(k_2+u)}\sum_{t\in \mathbb{Z}} (t+\kappa_{1,j})^{\nu-u}\frac{a_{j,l}(n+t+\kappa_{1,j}+\kappa_{2,l})\overline{b_j(t+\kappa_{1,j})}}{(n+t+\kappa_{1,j}+\kappa_{2,l})^{k_1+k_2+2\nu-1}}.
\end{aligned}
\end{equation*}
Thus, we complete the proof of Theorem \ref{thm 4}.
\end{proof}

\section{Case of Jacobi forms}\label{s : Jacobi forms}

In this section, we review the notion of Jacobi forms of real weight on $\mathrm{SL}_{2}(\mathbb{Z})$.
For details, we refer to \cite[Section 2]{LL}.  We extend the Petersson inner product for Jacobi cusp forms and highlight the result from Theorem \ref{thm 5}, which describes the decomposition of Jacobi forms in terms of theta functions.  We then extend the theory by defining the extended Rankin-Cohen brackets for pairs of Jacobi forms and modular forms, and we demonstrate that these brackets yield Jacobi cusp forms. We then relate these extended Rankin-Cohen brackets for Jacobi forms to the Rankin-Cohen brackets for the corresponding vector-valued modular forms via an isomorphism. Building on this, we investigate the Petersson pairing involving these extended brackets and Poincaré series in the context of the isomorphic vector-valued forms. Finally, we conclude this section by defining a linear map based on the holomorphic extended Rankin-Cohen bracket that explicitly describes its adjoint map.

Let $\Gamma^{J}:=\SL_2(\mathbb{Z})\rtimes \mathbb{Z}^{2}$ be the Jacobi group. 
Let $m$ be a positive integer and $k$ be a real number. 
Let $\chi$ be a multiplier system of weight $k$ on $\mathrm{SL}_{2}(\mathbb{Z})$.
Given $M:=\left[\sm a & b\\
c & d\esm, (\lambda,\mu) \right]\in \Gamma^{J}$ and a function $f:\mathbb{H}\times \mathbb{C}\to \mathbb{C}$, we define $f|_{k,m,\chi}M$ by
\begin{equation*}
    \begin{aligned}
    \left(f|_{k,m,\chi}M\right)(\tau,z):=&\chi^{-1}\left(\left(\begin{smallmatrix}
    a & b\\
    c & d
    \end{smallmatrix}\right)\right)f\left(\frac{a\tau+b}{c\tau+d}, \frac{z+\lambda \tau+\mu}{c\tau+d}\right)\\
    &\times(c\tau+d)^{-k}\exp\left(2\pi i m\left(-\frac{c(z+\lambda\tau+\mu)^2}{c\tau+d}+\lambda^2\tau+2\lambda z\right)\right).
    \end{aligned}
\end{equation*}

\begin{dfn}\label{def : def of Jacobi form}
    A function $f:\mathbb{H}\times \mathbb{C}\to \mathbb{C}$ is called a Jacobi form of weight $k$, index $m$, and multiplier system $\chi$ on $\mathrm{SL}_{2}(\mathbb{Z})$ if it is holomorphic in $\mathbb{H}\times \mathbb{C}$ and satisfies the following conditions: 
    \begin{enumerate}
        \item   $\left(f|_{k,m,\chi}[\gamma,X]\right)(\tau,z)=f(\tau,z)$ for every $[\gamma,X]\in \Gamma^{J}$, and
        \item there is a real number $\kappa\in [0,1)$ such that 
        \[ f(\tau,z)=\sum_{\substack{n,r\in \ZZ \\ r^2-4m(n+\kappa)\leq 0}}c(n+\kappa,r)e^{2\pi i(n+\kappa)\tau}e^{2\pi i rz}. \]
    \end{enumerate}
\end{dfn}
A Jacobi form $f$ is called a Jacobi cusp form if $c(n+\kappa,r)=0$ for all $(n,r)\in \mathbb{Z}^{2}$ such that $r^2-4m(n+\kappa)=0$. 
Let $J_{k,m,\chi}$ (resp. $J^{cusp}_{k,m,\chi}$) denote the space of Jacobi forms (resp. Jacobi cusp forms) of weight $k$, index $m$, and multiplier system $\chi$ on $\mathrm{SL}_{2}(\mathbb{Z})$.
In the following, we define the Petersson inner product of Jacobi cusp forms. 

\begin{dfn}\label{def : def of Petersson inner product Jacobi}
Assume that $f$ and $g$ are in $J^{cusp}_{k,m,\chi}$. 
The Petersson inner product $\langle f, g \rangle$ of $f$ and $g$ is defined by 
\[ \langle f , g \rangle := \int_{\Gamma^{J}\backslash \mathbb{H}\times \mathbb{C}} f(\tau,z)\overline{g(\tau,z)} e^{-\frac{4\pi m v^2}{y}}y^{k} \frac{dxdydudv}{y^3},  \]
where $\tau=x+iy\in \mathbb{H}$ and $z=u+iv\in \mathbb{C}$. 
\end{dfn}

For an integer $\mu$ with $\mu\in \{1,\dots, 2m\}$, we define $\theta_{m,\mu}:\mathbb{H}\times \mathbb{C}\to \mathbb{C}$ by 
\[ \theta_{m,\mu}(\tau,z):=\sum_{\substack{r\in \ZZ \\ r\equiv \mu \pmod{2m}}} e^{2\pi i \frac{r^2}{4m}\tau} e^{2\pi i r z}.\]
An important connection between Jacobi forms and vector-valued modular forms is provided by the theta decomposition.
This decomposition expresses a Jacobi form as a linear combination of theta functions $\theta_{m,\mu}$ with coefficients that are functions on $\mathbb{H}$.
The following theorem describes how a Jacobi form gives rise to a vector-valued modular form.

\begin{thm}\label{thm 5}\cite[Section 5]{EZ1}
    Assume that $f\in J_{k,m,\chi}$. 
    Then, there is a unique tuple of functions $\{f_{\mu}\}_{\mu\in \{1,\dots, 2m\}}$ such that $f_{\mu} : \mathbb{H}\to \mathbb{C}$ for each $\mu\in \{1,\dots, 2m\}$ and that 
    \[ f(\tau,z)=\sum_{\mu=1}^{2m} f_{\mu}(\tau) \cdot \theta_{m,\mu}(\tau,z). \]
    Moreover, $F(\tau):=\sum_{\mu=1}^{2m} f_{\mu}(\tau) \mathbf{e_{\mu}} : \mathbb{H}\to \mathbb{C}^{2m}$ is a vector-valued modular form of weight $k-\frac{1}{2}$ and multiplier system $\chi'$ with respect to $\rho'$ on $\mathrm{SL}_{2}(\mathbb{Z})$. 
    Here, $\{\mathbf{e_1},\dots, \mathbf{e_{2m}}\}$ is the standard basis of $\mathbb{C}^{2m}$.
    (For the definitions of $\chi'$ and $\rho'$, see \cite[Section 2]{CL})
\end{thm}

By Theorem \ref{thm 5}, we can define a map $\Psi:J_{k,m,\chi}\to M_{k-\frac{1}{2},\chi',\rho'}$ by
\[ \Psi(f)(\tau):=\sum_{\mu=1}^{2m} f_{\mu}(\tau) \mathbf{e_{\mu}}, \]
where $f(\tau,z)=\sum_{\mu=1}^{2m}f_{\mu}(\tau)\cdot \theta_{m,\mu}(\tau,z)$.
Note that $\Psi$ is an isomorphism.
Moreover, under the isomorphism $\Psi$, the Petersson inner products on Jacobi forms and vector-valued modular forms are compatible up to a constant factor.  
By \cite[Theorem 5.3]{EZ1}, if $f,g\in J^{cusp}_{k,m,\chi}$, then 
\[ \langle f,g \rangle = \frac{1}{\sqrt{2m}} \langle  \Psi(f), \Psi(g) \rangle.  \]

The Rankin-Cohen bracket is extended to pairs of a Jacobi form and a modular form by using the operator
\[ L_m := -\frac{1}{4\pi^2}\left(8\pi i m \frac{\partial}{\partial \tau} - \frac{\partial^2}{\partial z^2} \right),\]
where $m$ is a positive integer (for example, see \cite{K}).

\begin{dfn}\cite[Definition 1.1]{K}\label{def 3}
    Let $f$ be a function on $\mathbb{H}\times \mathbb{C}$ and $g$ be a function on $\mathbb{H}$. 
    Let $k$ and $\ell$ be real numbers, and $m$ and $\nu$ be positive integers. 
    The extended $\nu$-th Rankin-Cohen brackets $[f,g]^{\mathrm{hol}}_{k,m,l,\nu}$ is defined by 
    \[[f,g]^{\mathrm{hol}}_{k,m,l,\nu}(\tau,z):=\sum_{r+s=\nu}\binom{\nu+k-3/2}{s}\binom{\nu+l-1}{r}(-4m)^s L_m^{r}(f(\tau,z))D^s_{\tau}(g(\tau)). \]
\end{dfn}

Similar to the case of vector-valued modular forms, the following proposition shows that the extended Rankin-Cohen bracket of a Jacobi form and a modular form is a Jacobi cusp form.

\begin{prop}\cite[Proposition 1.3]{K}\label{prop 1}
    Let $\nu$ be a positive integer. 
    Let $k$ and $l$ be real numbers.
    Let $\chi_1$ (resp. $\chi_2$) be a multiplier system of weight $k$ (resp. $l$) on $\mathrm{SL}_{2}(\mathbb{Z})$.
    Assume that $f\in J_{k,m,\chi_1}$ and that $g$ is a modular form of weight $l$ and multiplier system $\chi_2$ on $\mathrm{SL}_{2}(\mathbb{Z})$. 
    Then, we have 
    \[[f,g]^{\mathrm{hol}}_{k,m,l,\nu}\in J^{cusp}_{k+l+2\nu,m,\chi_1\chi_2}. \]
\end{prop}
\begin{proof}
    In \cite{K}, Kimura showed that 
    \begin{equation}\label{eq 5}
        [f,g]^{\mathrm{hol}}_{k,m,l,\nu}\in J_{k+l+2\nu,m,\chi_{1}\chi_2}, 
    \end{equation}
    when $k$ and $l$ are positive integers and both $\chi_1$ and $\chi_2$ are trivial. 
    By following the argument in the proof of  \cite[Proposition 1.3]{K}, one can also obtain \eqref{eq 5} for real numbers $k,l$ and arbitrary multiplier systems $\chi_1$ and $\chi_2$.
    Thus, to complete the proof of Proposition \ref{prop 1}, it remains to show that $[f,g]^{\mathrm{hol}}_{k,m,l,\nu}$ is a cusp form.

    Let 
    \[ f(\tau,z)=\sum_{\substack{n,r\in \ZZ \\ r^2-4m(n+\kappa_1)\leq 0}}c_{f}(n+\kappa_1,r) e^{2\pi i (n+\kappa_1)\tau}e^{2\pi i r z} \]
    and 
    \[ g(\tau)=\sum_{\substack{n\in \ZZ \\ n+\kappa_2\geq 0}} a_{g}(n+\kappa_2) e^{2\pi i (n+\kappa_2)\tau}, \]
    where $\kappa_1,\kappa_2\in [0,1)$ are real numbers. 
    Since
    \begin{equation*}
            L_m\left(e^{2\pi i (n+\kappa_1)\tau} e^{2\pi i r z} \right)=-4(r^2-4m(n+\kappa_1))e^{2\pi i (n+\kappa_1)\tau} e^{2\pi i r z},
    \end{equation*}
    we obtain
    \begin{equation*}
        \begin{aligned}
            [f,g]^{\mathrm{hol}}_{k,m,l,\nu}(\tau,z)&=\sum_{i=0}^{\nu}\binom{\nu+k-3/2}{\nu-i}\binom{\nu+l-1}{i}(-4m)^{\nu-i}L_m^{i}(f(\tau,z)) D^{\nu-i}_{\tau}(g(\tau))\\
            &=\sum_{i=0}^{\nu}\binom{\nu+k-3/2}{\nu-i}\binom{\nu+l-1}{i}(-4m)^{\nu-i}\\
            &\times\left(\sum_{\substack{n,r\in \ZZ \\ r^2-4m(n+\kappa_1)\leq 0}}(-4(r^2-4m(n+\kappa_1)))^{i}c_{f}(n+\kappa_1,r) e^{2\pi i (n+\kappa_1)\tau}e^{2\pi i r z} \right)\\
            &\times \left(\sum_{\substack{n\in \ZZ\\ n+\kappa_2\geq 0}} (n+\kappa_2)^{\nu-i} a_{g}(n+\kappa_2) e^{2\pi i (n+\kappa_2)\tau} \right).
        \end{aligned}
    \end{equation*}
    Then, $[f,g]^{\mathrm{hol}}_{k,m,l,\nu}(\tau,z)$ can be expressed as 
    \begin{equation*}
        [f,g]^{\mathrm{hol}}_{k,m,l,\nu}(\tau,z) = \sum_{\substack{r,n\in \ZZ\\ r^2-4m(n+\kappa_1+\kappa_2)\leq 0}} c(n+\kappa_1+\kappa_2,r) e^{2\pi i (n+\kappa_1+\kappa_2)\tau}e^{2\pi i r z},
    \end{equation*}
    where the coefficients $c(n+\kappa_1+\kappa_2,r)$ are given by
    \begin{equation*}
        \begin{aligned}
            c(n+\kappa_1+\kappa_2,r)&=\sum_{i=0}^{\nu} \binom{\nu+k-3/2}{\nu-i}\binom{\nu+l-1}{i}(-4m)^{\nu-i}\\
            &\times\sum_{\substack{n_1+n_2=n\\r^2-4m(n_1+\kappa_1)\leq 0\\ n_2+\kappa_2\geq 0}}(-4(r^2-4m(n_1+\kappa_1)))^{i}c_{f}(n_1+\kappa_1,r)(n_2+\kappa_2)^{\nu-i}a_{g}(n_2+\kappa_2).
        \end{aligned}
    \end{equation*}
    Assume that $r^2-4m(n+\kappa_1+\kappa_2)=0$ and let $n=n_1+n_2$ with $r^2-4m(n_1+\kappa_1)\leq 0$ and $n_2+\kappa_2\geq 0$.
    Then it follows that 
    \[ r^2-4m(n_1+\kappa_1)=0=n_2+\kappa_2. \]
    Thus, $c(n+\kappa_1+\kappa_2,r)=0$ whenever $r^2-4m(n+\kappa_1+\kappa_2)=0$. 
    Therefore, $[f,g]^{\mathrm{hol}}_{k,m,l,\nu}$ is a cusp form.
\end{proof}

The following theorem demonstrates the compatibility between the extended Rankin–Cohen bracket on Jacobi forms and the classical Rankin–Cohen bracket on vector-valued modular forms via the theta decomposition.

\begin{thm}\label{thm 7}
    Let $l$ be a real number.
    Assume that $f\in J_{k,m,\chi}$ and $g$ is a modular form of weight $\ell$ and multiplier system $\chi_1$ on $\mathrm{SL}_{2}(\mathbb{Z})$. 
    For a positive integer $\nu$, we have 
    \[  \frac{1}{(4m)^{\nu}}\cdot\Psi\left([f,g]^{\mathrm{hol}}_{k,m,l,\nu}\right) = \frac{1}{\nu !}\cdot[\Psi(f),g]_{\nu}. \]
\end{thm}

\begin{proof}
    Note that for any integer $a$, we have 
    \begin{equation*}
        \begin{aligned}
            L_m\left(e^{\frac{\pi i a^2\tau}{2m}}e^{2\pi i a z}\right)&=-\frac{1}{4\pi^2}\left(8\pi i m \frac{\partial}{\partial \tau} - \frac{\partial^2}{\partial z^2}\right)\left(e^{\frac{\pi i a^2\tau}{2m}}e^{2\pi i a z}\right)=0.
        \end{aligned}
    \end{equation*}
    Thus, we have 
    \begin{equation}\label{eq 2}
    \begin{aligned}
        L_m\left(f\right)&=L_m\left(\sum_{\mu=1}^{2m} f_{\mu}(\tau)\cdot \theta_{m,\mu}(\tau,z)\right)\\
        &=\sum_{\mu=1}^{2m} L_m\left(f_{\mu}(\tau)\cdot \theta_{m,\mu}(\tau,z) \right)\\
        &=\sum_{\mu=1}^{2m} L_m\left(f_{\mu}(\tau)\right)\cdot \theta_{m,\mu}(\tau,z)\\
        &=4m \sum_{\mu=1}^{2m} D_{\tau}(f_{\mu}(\tau))\cdot \theta_{m,\mu}(\tau,z).
    \end{aligned}
    \end{equation}
    By \eqref{eq 2}, for any non-negative integer $r$, we have 
    \[ L_m^{r}\left(f\right) = (4m)^{r}\sum_{\mu=1}^{2m} D_{\tau}^{r}(f_{\mu}(\tau))\cdot \theta_{m,\mu}(\tau,z).  \]
    Then, we get 
    \begin{equation*}
        \begin{aligned}
            &[f,g]^{\mathrm{hol}}_{k,m,l,\nu}\\
            &=\sum_{r+s=\nu}\binom{\nu+k-3/2}{s}\binom{\nu+l-1}{r} (-4m)^{s}\sum_{\mu=1}^{2m} (4m)^{r} D_{\tau}^{r}(f_{\mu}(\tau))\cdot \theta_{m,\mu}(\tau,z) D_{\tau}^{s}(g(\tau))\\
            &=\sum_{\mu=1}^{2m}\left((4m)^{\nu}\sum_{r+s=\nu}(-1)^{s}\binom{\nu+k-3/2}{s}\binom{\nu+l-1}{r} D_{\tau}^{r}(f_{\mu}(\tau))D_{\tau}^{s}(g(\tau))\right)\cdot \theta_{m,\mu}(\tau,z).
        \end{aligned}
    \end{equation*}
    Thus, we obtain 
    \begin{equation*}
        \begin{aligned}
            &\Psi\left([f,g]^{\mathrm{hol}}_{k,m,l,\nu}\right)\\
            &= (4m)^{\nu}\cdot\sum_{\mu=1}^{2m}\left(\sum_{r+s=\nu}(-1)^{s}\binom{\nu+k-3/2}{s}\binom{\nu+l-1}{r} D_{\tau}^{r}(f_{\mu}(\tau))D_{\tau}^{s}(g(\tau))\right)\mathbf{e_{\mu}}.
        \end{aligned}
    \end{equation*}

    Note that the weight of $\Psi(f)$ is $k-\frac{1}{2}$. 
    Hence, we have
    \begin{equation*}
        \begin{aligned}
            [\Psi(f),g]_{\nu}&=\sum_{i=0}^{\nu} (-1)^{\nu-i} \binom{\nu}{i}\frac{\Gamma(k-\frac{1}{2}+\nu)\Gamma(l+\nu)}{\Gamma(k-\frac{1}{2}+i)\Gamma(l+\nu-i)}\Psi(f)^{(i)}(\tau)\otimes g^{(\nu-i)}(\tau)\\
            &=\nu!\sum_{r+s=\nu} (-1)^{s} \binom{\nu+k-3/2}{s}\binom{\nu+l-1}{r} \sum_{\mu=1}^{2m} D_{\tau}^{r}(f_{\mu}(\tau))D_{\tau}^{s}(g(\tau)) \mathbf{e_{\mu}}\\
            &=\nu!\cdot\sum_{\mu=1}^{2m}\left(\sum_{r+s=\nu} (-1)^{s} \binom{\nu+k-3/2}{s}\binom{\nu+l-1}{r}  D_{\tau}^{r}(f_{\mu}(\tau))D_{\tau}^{s}(g(\tau))\right)\mathbf{e_{\mu}}.
        \end{aligned}
    \end{equation*}
    Therefore, we completed the proof of Theorem \ref{thm 7}.
\end{proof}

Using the compatibility established in Theorem \ref{thm 7}, results from the theory of vector-valued modular forms can be translated into the setting of Jacobi forms.
As a consequence, Theorems \ref{thm 3} and \ref{thm 4} naturally induce analogous statements for Jacobi forms.
To formulate the corresponding results of Jacobi forms, we begin by recalling the definition of the Poincaré series for modular forms.
 For a non-negative integer $s$, let 
    \[ \mathbb{P}(\tau;k_2,\chi_2,s):=\frac{1}{2} \sum_{M\in \mathrm{SL}_{2}(\mathbb{Z})/\langle T \rangle } \chi_2^{-1}(M)(c\tau+d)^{-k_2}\mathrm{exp}(2\pi i(s+\kappa)M\tau), \]
    where $M=\sm a & b\\
    c & d \esm$ and $\chi_2(T)=e^{2\pi i \kappa}$ with $\kappa\in [0,1)$.
The following theorem gives the Jacobi form analogue of Theorem \ref{thm 3}.
    
\begin{thm}\label{thm 9}
    For each $i\in \{1,2\}$, let $k_i$ be a real number with $k_i>2$ and $\chi_i$ be a multiplier system of weight $k_i$ on $\mathrm{SL}_{2}(\mathbb{Z})$.
    Let $m$ and $\nu$ be positive integers. 
    Let $\kappa_i\in [0,1)$ be a real number such that $\chi_{i}(T)=e^{2\pi i \kappa_i}$.
    Assume that 
    \[ f(\tau,z)=\sum_{\substack{n,r\in \mathbb{Z}\\r^2-4m(n+\kappa_1+\kappa_2)\leq 0}} c_{f}(n+\kappa_1+\kappa_2,r) e^{2\pi i (n+\kappa_1+\kappa_2)\tau}e^{2\pi i r z}\in J_{k_1+k_2+2\nu,m,\chi_1\chi_2} \]
    and that
    \[ g(\tau,z)=\sum_{\substack{n,r\in \mathbb{Z}\\r^2-4m(n+\kappa_1)\leq 0}} c_{g}(n+\kappa_1,r) e^{2\pi i (n+\kappa_1)\tau}e^{2\pi i r z}\in J_{k_1,m,\chi_1}. \]
    Then, we have 
    \begin{equation*}
        \begin{aligned}
            &\langle f, [g,\mathbb{P}(\tau;k_2,\chi_2,s)]_{k_1,m,k_2,\nu}^{\mathrm{hol}} \rangle\\
            &=\frac{(4m)^{\nu}\Gamma(k_1+k_2+2\nu-\frac{3}{2})}{\sqrt{2m}\cdot \nu!\cdot (4\pi)^{k_1+k_2+2\nu-1}}\sum_{j=1}^{2m}\sum_{u=0}^{\nu} (-(s+\kappa_{2}))^{u}\binom{\nu}{u}\\
            &\times\frac{\Gamma(k_1+\nu-\frac{1}{2})\Gamma(k_2+\nu)}{\Gamma(k_1+\nu-u-\frac{1}{2})\Gamma(k_2+u)}\sum_{n\in \mathbb{Z}} \left(n+\kappa_{1}-\frac{j^2}{4m}\right)^{\nu-u}\frac{c_{f}(s+n+\kappa_{1}+\kappa_{2},j)\overline{c_g(n+\kappa_{1},j)}}{\left(s+n+\kappa_{1}+\kappa_{2}-\frac{j^2}{4m}\right)^{k_1+k_2+2\nu-1}}.
        \end{aligned}
    \end{equation*}
\end{thm}

\begin{proof}
    By applying Theorem \ref{thm 7}, we obtain
    \begin{equation*}
        \begin{aligned}
            \langle f, [g,\mathbb{P}(\tau;k_2,\chi_2,s)]_{k_1,m,k_2,\nu}^{\mathrm{hol}} \rangle&=\frac{1}{\sqrt{2m}}\langle \Psi(f), \Psi([g,\mathbb{P}(\tau;k_2,\chi_2,s)]^{\mathrm{hol}}_{k_1,m,k_2,\nu}) \rangle\\
            &=\frac{(4m)^{\nu}}{\nu ! \cdot \sqrt{2m}}\langle \Psi(f), [\Psi(g),\mathbb{P}(\tau;k_2,\chi_2,s)]_{\nu} \rangle.
        \end{aligned}
    \end{equation*}
    Note that $\Psi(g)\in M_{k_1-\frac{1}{2},\chi_1\chi',\rho'}$ and that the $(j,j)$-entry of $\chi'(T)\rho'(T)$ is equal to $e^{-2\pi i \frac{j^2}{4m}}$.
    By the theta expansion, we have 
    \begin{equation*}
        \Psi(g)(\tau)=\sum_{\mu=1}^{2m} \left(\sum_{\substack{n\in \ZZ\\ \mu^2-4m(n+\kappa_1)\leq 0}} c_g(n+\kappa_1,\mu) \exp\left(2\pi i\tau \left(n+\kappa_1-\frac{\mu^2}{4m} \right)\right)\right) \mathbf{e_{\mu}}.
    \end{equation*}
    Similarly, we obtain $\Psi(f)\in M_{k_1+k_2-\frac{1}{2},\chi_1\chi_2\chi',\rho'}$ and 
    \begin{equation*}
        \Psi(f)(\tau)=\sum_{\mu=1}^{2m} \left(\sum_{\substack{n\in \ZZ\\ \mu^2-4m(n+\kappa_1+\kappa_2)\leq 0}} c_f(n+\kappa_1+\kappa_2,\mu) \exp\left(2\pi i\tau \left(n+\kappa_1+\kappa_2-\frac{\mu^2}{4m} \right)\right)\right) \mathbf{e_{\mu}}.
    \end{equation*}
    Using Theorem \ref{thm 3}, we conclude the proof of  Theorem \ref{thm 9}.
\end{proof}

Similarly, using the correspondence between Jacobi-forms and vector-valued modular forms, we obtain an explicit description of the adjoint map of the linear map induced by the extended Rankin-Cohen bracket.
The following theorem provides the Jacobi-form analogue of Theorem \ref{thm 4}.

\begin{thm}\label{prop 2}
For each $i\in \{1,2\}$, let $k_i$, $\chi_i$ and $\kappa_i$ be as in Theorem \ref{thm 9}.
Let $m$ and $\nu$ be positive integers. 
Assume that $g(\tau)=\sum_{n\in \mathbb{Z}} b(n+\kappa_1)e^{2\pi i (n+\kappa_1)\tau}$ is a modular form of weight $k_1$ and multiplier system $\chi_1$. 
Let $T^{\mathrm{hol}}_{g,\nu}: J^{cusp}_{k_2,m,\chi_2}\to J^{cusp}_{k_1+k_2+2\nu,m,\chi_1\chi_2}$ be defined by $T^{\mathrm{hol}}_{g,\nu}(f):=[f,g]^{\mathrm{hol}}_{k_2,m,k_1,\nu}$. 
If 
\[h(\tau,z)=\sum_{\substack{n,r\in \ZZ \\ r^2-4m(n+\kappa_1+\kappa_2)\leq 0}} c_h(n+\kappa_1+\kappa_2,r)e^{2\pi i (n+\kappa_1+\kappa_2)\tau}e^{2\pi i r z}\in J^{cusp}_{k_1+k_2+2\nu,m,\chi_1\chi_2},\]
then   
\[ \left(T^{\mathrm{hol},*}_{g,\nu}(h)\right)(\tau,z)=\sum_{\substack{n,r\in \ZZ\\ r^2-4m(n+\kappa_2)\leq 0}} c(n+\kappa_2,r)e^{2\pi i (n+\kappa_2)\tau}e^{2\pi i r z},\]
where $\mu\in\{1,\dots,2m\}$ with $\mu\equiv r\pmod{2m}$ and
\begin{equation*}
        \begin{aligned}
            &c(n+\kappa_2,r)\\
            &=\frac{(-4m)^{\nu}}{\nu!}(4\pi(n+\kappa_{2}-\frac{\mu^2}{4m}))^{k_2-\frac{3}{2}}\frac{\Gamma(k_1+k_2+2\nu-\frac{3}{2})}{(4\pi)^{k_1+k_2+2\nu-\frac{3}{2}}}\sum_{u=0}^{\nu} (-(n+\kappa_{2}-\frac{\mu^2}{4m}))^{u}\binom{\nu}{u}\\
            &\times\frac{\Gamma(k_1+\nu)\Gamma(k_2+\nu-\frac{1}{2})}{\Gamma(k_1+\nu-u)\Gamma(k_2+u-\frac{1}{2})\Gamma(k_2-\frac{3}{2})}\sum_{t\in \mathbb{Z}} (t+\kappa_{1})^{\nu-u}\frac{c_h(n+t+\kappa_{1}+\kappa_{2},\mu)\overline{b(t+\kappa_{1})}}{(n+t+\kappa_{1}+\kappa_{2}-\frac{\mu^2}{4m})^{k_1+k_2+2\nu-\frac{3}{2}}}.
        \end{aligned}
    \end{equation*}
Here, $T^{\mathrm{hol},*}_{g,\nu}$ denotes the adjoint map of $T^{\mathrm{hol}}_{g,\nu}$.
\end{thm}

\begin{proof}
    Note that by the theta expansion, we have 
    \begin{equation*}
        \Psi(h)(\tau)=\sum_{\mu=1}^{2m} \left(\sum_{\substack{n\in \ZZ\\ \mu^2-4m(n+\kappa_1+\kappa_2)\leq 0}} c_h(n+\kappa_1+\kappa_2,\mu) \exp\left(2\pi i\tau \left(n+\kappa_1+\kappa_2-\frac{\mu^2}{4m} \right)\right)\right) \mathbf{e_{\mu}}.
    \end{equation*}
    To prove Proposition \ref{prop 2}, it is enough to show that
    \[ T^{\mathrm{hol},*}_{g,\nu}(h)=\frac{(4m)^{\nu}}{\nu!}\Psi^{-1}(T^{*}_{g,\nu}(\Psi(h))).\]
    Since $\Psi$ is an isomorphism, we have 
    \begin{equation*}
        \begin{aligned}
            \langle \Psi(T^{\mathrm{hol},*}_{g,\nu}(h)), \Psi(f) \rangle &= \sqrt{2m}\langle  T^{\mathrm{hol},*}_{g,\nu}(h),f \rangle \\
            &=\sqrt{2m} \langle h, T^{\mathrm{hol}}_{g,\nu}(f) \rangle\\
            &= \sqrt{2m} \langle h, [f,g]^{\mathrm{hol}}_{k_2,m,k_1,\nu} \rangle\\
            &=\langle \Psi(h), \Psi([f,g]^{\mathrm{hol}}_{k_2,m,k_1,\nu}) \rangle\\
            &=\langle \Psi(h), \frac{(4m)^v}{v!} [\Psi(f),g]_{\nu} \rangle \\
            &= \frac{(4m)^{\nu}}{\nu !}\langle \Psi(h), T_{g,\nu}(\Psi(f)) \rangle \\
            &=\langle \frac{(4m)^{\nu}}{\nu!}T^{*}_{g,\nu}(\Psi(h)), \Psi(f)  \rangle.
        \end{aligned}
    \end{equation*}
\end{proof}

\section{Case of Skew-holomorphic Jacobi forms}\label{s : skew Jacobi}

In this section, we review the notion of skew-holomorphic Jacobi forms of real weight on $\mathrm{SL}_{2}(\mathbb{Z})$
and establish analogues of the results in Section \ref{s : Jacobi forms} in the setting of skew-holomorphic Jacobi forms.
Similar to the case of Jacobi forms, we introduce an isomorphism between the space of skew-holomorphic Jacobi forms and that of vector-valued modular forms via the theta expansion. 
This isomorphism allows us to establish the compatibility of the Petersson inner product and the Rankin–Cohen bracket between vector-valued modular forms and skew-holomorphic Jacobi forms.
From this compatibility, analogous results for Jacobi forms, as presented in Section \ref{s : Jacobi forms}, can be obtained through parallel arguments. For the reader’s convenience, we state the corresponding theorems explicitly without proof.
We follow the notation introduced in Section \ref{s : Jacobi forms}. 

For a function $f:\mathbb{H}\times \mathbb{C}\to \mathbb{C}$ and $M=[\sm a & b\\
c & d \esm , (\lambda,\mu)]\in \Gamma^{J}$, we define $f|^{sk}_{k,m,\chi}M$ by
\begin{equation*}
    \begin{aligned}
    \left(f|^{sk}_{k,m,\chi}M\right)(\tau,z)
    :=&\chi^{-1}\left(\left(\begin{smallmatrix}
    a & b\\
    c & d
    \end{smallmatrix}\right)\right)f\left(\frac{a\tau+b}{c\tau+d}, \frac{z+\lambda \tau+\mu}{c\tau+d}\right)\\ &\times(c\overline{\tau}+d)^{1-k}|c\tau+d|^{-1}\exp\left(2\pi i m\left(-\frac{c(z+\lambda\tau+\mu)^2}{c\tau+d} +\lambda^2\tau+2\lambda z\right)\right).
    \end{aligned}
\end{equation*}

\begin{dfn}\label{def : def of skew-Jacobi form}
    A function $f:\mathbb{H}\times \mathbb{C}\to \mathbb{C}$ is called a skew-holomorphic Jacobi form of weight $k$, index $m$, and multiplier system $\chi$ on $\mathrm{SL}_{2}(\mathbb{Z})$ if it is real analytic in $\mathbb{H}$ and holomorphic in $\mathbb{C}$, and $f$ satisfies the following conditions:
    \begin{enumerate}
        \item $\left(f|^{sk}_{k,m,\chi}[\gamma,X]\right)(\tau,z)=f(\tau,z)$ for every $[\gamma,X]\in \Gamma^{J}$, and 
        \item there is a real number $\kappa\in [0,1)$ such that 
        \[ f(\tau,z)= \sum_{\substack{n,r\in \ZZ \\ r^2-4m(n+\kappa)\geq 0}}d(n+\kappa,r)\exp\left(\frac{-\pi y}{m}\left(r^2-4m(n+\kappa)\right)\right)e^{2\pi i(n+\kappa)\tau}e^{2\pi i rz},\]
        where $\tau=x+iy$.
    \end{enumerate}
\end{dfn}

A skew-holomorphic Jacobi form $f$ is called a skew-holomorphic Jacobi cusp form if $d(n+\kappa,r)=0$ for all $(n,r)\in \mathbb{Z}^2$ with $r^2-4m(n+\kappa)=0$. 
Let $J^{sk}_{k,m,\chi}$ (resp. $J^{sk,cusp}_{k,m,\chi}$) denote the space of skew-holomorphic Jacobi forms (resp. skew-holomorphic Jacobi cusp forms) of weight $k$, index $m$, and multiplier system $\chi$ on $\SL_2(\mathbb{Z})$.
Now, we introduce the definition of the Petersson inner product of skew-holomorphic Jacobi cusp forms. 

\begin{dfn}\label{def : def of Petersson inner product Skew - Jacobi}
Let $k$ be a real number and $m$ be a positive integer. 
Let $\chi$ be a multiplier system of weight $k$ on $\mathrm{SL}_{2}(\mathbb{Z})$. 
Assume that $f$ and $g$ are in $J^{sk,cusp}_{k,m,\chi}$. 
The Petersson inner product $\langle f, g \rangle$ of $f$ and $g$ is defined by 
\[ \langle f , g \rangle := \int_{\Gamma^{J}\backslash \mathbb{H}\times \mathbb{C}} f(\tau,z)\overline{g(\tau,z)} e^{-\frac{4\pi m v^2}{y}}y^{k} \frac{dxdydudv}{y^3},  \]
where $\tau=x+iy\in \mathbb{H}$ and $z=u+iv\in \mathbb{C}$. 
\end{dfn}

Similar to Jacobi forms, skew-holomorphic Jacobi forms correspond to vector-valued modular forms via the theta expansion. 
The following theorem describes the correspondence between skew-holomorphic Jacobi forms and vector-valued modular forms. 

\begin{thm}\label{thm 6}\cite[Section 6]{BR}
    Assume that $g\in J^{sk}_{k,m,\chi}$. 
    Then, there is a unique tuple of functions $\{g_{\mu}\}_{\mu\in \{1,\dots, 2m\}}$ such that $f_{\mu} : \mathbb{H}\to \mathbb{C}$ for each $\mu\in \{1,\dots, 2m\}$ and that 
    \[ g(\tau,z)=\sum_{\mu=1}^{2m} g_{\mu}(\tau) \cdot \theta_{m,\mu}(\tau,z). \]
    Moreover, $G(\tau):=\sum_{\mu=1}^{2m} \overline{g_{\mu}(\tau)} \mathbf{e_{\mu}} : \mathbb{H}\to \mathbb{C}^{2m}$ is a vector-valued modular form of weight $k-\frac{1}{2}$ and multiplier system $\overline{\chi'}$ with respect to $\overline{\rho'}$ on $\mathrm{SL}_{2}(\mathbb{Z})$. 
\end{thm}

The theta expansion induces an isomorphism $\Psi^{sk} : J_{k,m,\chi}^{sk}\to M_{k-\frac{1}{2},\overline{\chi'},\overline{\rho'}}$ given by 
\[ \Psi^{sk}(g)(\tau):=\sum_{\mu=1}^{2m} \overline{g_{\mu}(\tau)}\mathbf{e_{\mu}}, \]
where $g(\tau,z)=\sum_{\mu=1}^{2m}g_{\mu}(\tau)\cdot \theta_{m,\mu}(\tau,z)$.
Following the argument in \cite[Theorem 5.3]{EZ1}, if $f,g\in J^{sk,cusp}_{k,m,\chi}$, then we have
\[ \langle f,g \rangle = \frac{1}{\sqrt{2m}}\overline{\langle \Psi^{sk}(f), \Psi^{sk}(g) \rangle}.\]

Analogous to the Jacobi form case, one also defines the extended Rankin–Cohen bracket of a skew-holomorphic Jacobi form with a modular form (for example, see \cite{K}).

\begin{dfn}\cite[Definition 1.1]{K}\label{def 4}
    Let $f$ be a function on $\mathbb{H}\times \mathbb{C}$ and $g$ be a function on $\mathbb{H}$. 
    Let $k$ and $\ell$ be real numbers, and $\nu$ be a positive integer. 
    The extended $\nu$-th Rankin-Cohen bracket $[f,g]^{\mathrm{skew}}_{k,l,\nu}$ is defined by 
    \[ [f,g]^{\mathrm{skew}}_{k,l,\nu}(\tau,z):=\sum_{r+s=\nu}(-1)^r\binom{\nu+k-3/2}{s}\binom{\nu+l-1}{r} D^r_{-\overline{\tau}}(f(\tau,z))D^s_{-\overline{\tau}}(\overline{g(\tau)}).  \]
\end{dfn}

The following proposition shows that the $n$-th extended Rankin-Cohen bracket of a skew-holomorphic Jacobi form with a modular form is a skew-holomorphic Jacobi cusp form. 

\begin{prop}\cite[Proposition 1.3]{K}\label{prop 4}
    Let $\nu$ be a positive integer.
    Let $k$ and $l$ be real numbers and $\chi_1$ (resp. $\chi_2$) be a multiplier system of weight $k$ (resp. $l$) on $\mathrm{SL}_{2}(\mathbb{Z})$. 
    Assume that $f\in J^{sk}_{k,m,\chi_1}$ and that
    $g$ is a modular form of weight $l$ and multiplier system $\chi_2$ on $\mathrm{SL}_{2}(\mathbb{Z})$.
    Then, we have 
    \[ [f,g]^{\mathrm{skew}}_{k,l,\nu}\in J^{sk,cusp}_{k+l+2\nu,m,\chi_1\overline{\chi_2}} \]
\end{prop}
\begin{proof}
    Similar to Proposition \ref{prop 1}, Kimura \cite{K} proved that 
    \[ [f,g]^{\mathrm{skew}}_{k,l,\nu}\in J^{sk}_{k_1+l+2\nu,m,\chi_1\overline{\chi_2}}, \]
    when $k$ and $l$ are positive integers and both $\chi_1$ and $\chi_2$ are trivial.
    Since the remaining part of the proof is analogous to that of Proposition \ref{prop 1} in the Jacobi form case, we omit the details.  
\end{proof}

Via the map $\Psi^{sk}$, the following theorem shows that the extended Rankin–Cohen bracket on skew-holomorphic Jacobi forms aligns, up to a constant multiple, with the  Rankin–Cohen bracket on vector-valued modular forms.

\begin{thm}\label{thm 8}
    Assume that $f\in J^{sk}_{k,m,\chi}$ and that $g$ is a modular form of weight $\ell$ and multiplier system $\chi_1$ on $\mathrm{SL}_{2}(\mathbb{Z})$. 
    For a positive integer $\nu$, we have 
    \[ \Psi^{sk}\left([f,g]^{\mathrm{skew}}_{k,l,\nu}\right) = \frac{1}{\nu !}\cdot[\Psi^{sk}(f),g]_{\nu}. \]
\end{thm}

\begin{proof}
    Since $\theta_{m,\mu}(\tau,z)$ is a holomorphic function in $\tau\in \mathbb{H}$, it follows that for any $\mu\in \{1,\dots,2m\}$, 
    \[ D_{-\overline{\tau}}(\theta_{m,\mu}(\tau,z))=0.\]
    Thus, for any non-negative integer $r$, we have 
    \[ D^{r}_{-\overline{\tau}}(f)=\sum_{\mu=1}^{2m} D^{r}_{-\overline{{\tau}}}(f_{\mu}(\tau))\cdot \theta_{m,\mu}(\tau,z). \]
    As in the proof of Theorem \ref{thm 7}, one can complete the rest of the proof of Theorem \ref{thm 8} by using the definitions of the extended Rankin-Cohen bracket $[f,g]_{k,l,\nu}^{\mathrm{skew}}$ and $\Psi^{sk}$. 
\end{proof}

As in the Jacobi form case, the compatibility provided by Theorem \ref{thm 8} allows us to reformulate Theorems \ref{thm 3} and \ref{thm 4} as results about skew‐holomorphic Jacobi forms. 
Thus, we state the corresponding analogue in the following theorems without proof, as the arguments are analogous to those used in the proofs of Theorems \ref{thm 9} and \ref{prop 2}.

\begin{thm}\label{thm 10}
    For each $i\in \{1,2\}$, let $k_i$ be a real number with $k_i>2$ and $\chi_i$ be a multiplier system of weight $k_i$ on $\mathrm{SL}_{2}(\mathbb{Z})$. 
    Let $m$ and $\nu$ be positive integers. 
    Let $\kappa_i\in [0,1)$ be a real number such that $\chi_{i}(T)=e^{2\pi i \kappa_i}$.
    Assume that $f\in J^{sk}_{k_1+k_2+2\nu,m,\chi_1\overline{\chi_2}}$ has a Fourier expansion of the form
    \[ f(\tau,z)=\sum_{\substack{n,r\in \mathbb{Z}\\r^2-4m(n+\kappa_1-\kappa_2)\geq 0}} d_{f}(n+\kappa_1-\kappa_2,r)\exp\left(-\frac{\pi y}{m}(r^2-4m(n+\kappa_1-\kappa_2))\right)e^{2\pi i (n+\kappa_1-\kappa_2)\tau}e^{2\pi i r z} \]
    and that $g\in J^{sk}_{k_1,m,\chi_1}$ has a Fourier expansion of the form
    \[ g(\tau,z)=\sum_{\substack{n,r\in \mathbb{Z}\\r^2-4m(n+\kappa_1)\geq 0}} d_{g}(n+\kappa_1,r)\exp\left(-\frac{\pi y}{m}(r^2-4m(n+\kappa_1))\right) e^{2\pi i (n+\kappa_1)\tau}e^{2\pi i r z}. \]
    Then, we have 
    \begin{equation*}
        \begin{aligned}
            &\langle f, [g,\mathbb{P}(\tau;k_2,\chi_2,s)]_{k_1,k_2,\nu}^{\mathrm{skew}} \rangle\\
            &=\frac{\Gamma(k_1+k_2+2\nu-\frac{3}{2})}{\sqrt{2m}\cdot \nu!\cdot (4\pi)^{k_1+k_2+2\nu-1}}\sum_{j=1}^{2m}\sum_{u=0}^{\nu} (-(s+\kappa_{2}))^{u}\binom{\nu}{u}\\
            &\times\frac{\Gamma(k_1+\nu-\frac{1}{2})\Gamma(k_2+\nu)}{\Gamma(k_1+\nu-u-\frac{1}{2})\Gamma(k_2+u)}\sum_{n\in \mathbb{Z}} \left(-n-\kappa_{1}+\frac{j^2}{4m}\right)^{\nu-u}\frac{d_{f}(-s+n+\kappa_{1}-\kappa_{2},j)\overline{d_g(n+\kappa_{1},j)}}{\left(s-n-\kappa_{1}+\kappa_{2}+\frac{j^2}{4m}\right)^{k_1+k_2+2\nu-1}}.
        \end{aligned}
    \end{equation*}
\end{thm}

\begin{thm}\label{thm 11}
    For each $i\in \{1,2\}$, let $k_i$, $\chi_i$ and $\kappa_i$ be as in Theorem \ref{thm 10}.
    Let $m$ and $\nu$ be positive integers.
    Assume that $g(\tau)=\sum_{n\in \mathbb{Z}} b(n+\kappa_1)e^{2\pi i (n+\kappa_1)\tau}$ is a modular form of weight $k_1$ and multiplier system $\chi_1$.
    Let $T^{\mathrm{skew}}_{g,\nu} : J^{sk,cusp}_{k_2,m,\chi_2}\to J^{sk,cusp}_{k_1+k_2+2\nu,m,\overline{\chi_1}\chi_2}$ be defined by $T^{\mathrm{skew}}_{g,\nu}(f):=[f,g]^{\mathrm{skew}}_{k_2,k_1,\nu}$. 
    If $h\in J^{sk,cusp}_{k_1+k_2+2\nu,m,\overline{\chi_1}\chi_2}$ has a Fourier expansion of the form
    \begin{equation*}
        \begin{aligned}
            &h(\tau,z)\\
            &=\sum_{\substack{n,r\in \mathbb{Z}\\r^2-4m(n-\kappa_1+\kappa_2)\geq 0}} d_{h}(n-\kappa_1+\kappa_2,r)\exp\left(\frac{-\pi y}{m}\left(r^2-4m(n-\kappa_1+\kappa_2)\right)\right) e^{2\pi i (n-\kappa_1+\kappa_2)\tau}e^{2\pi i r z},
        \end{aligned}
    \end{equation*}
    then, 
    \[ \left(T^{\mathrm{skew},*}_{g,\nu}(h)\right)(\tau,z)=\sum_{\substack{n,r\in \mathbb{Z}\\r^2-4m(n+\kappa_2)\geq 0}} d(n+\kappa_2,r)\exp\left(\frac{-\pi y}{m}\left(r^2-4m(n+\kappa_2)\right)\right) e^{2\pi i (n+\kappa_2)\tau}e^{2\pi i r z}, \]
    where $\mu\in\{1,\dots,2m\}$ with $\mu\equiv r\pmod{2m}$ and
    \begin{equation*}
        \begin{aligned}
            &d(n+\kappa_2,r)\\
            &=\frac{(-1)^{\nu}}{\nu!}(4\pi(-n-\kappa_{2}+\frac{\mu^2}{4m}))^{k_2-\frac{3}{2}}\frac{\Gamma(k_1+k_2+2\nu-\frac{3}{2})}{(4\pi)^{k_1+k_2+2\nu-\frac{3}{2}}}\sum_{u=0}^{\nu} (n+\kappa_{2}-\frac{\mu^2}{4m})^{u}\binom{\nu}{u}\\
            &\times\frac{\Gamma(k_1+\nu)\Gamma(k_2+\nu-\frac{1}{2})}{\Gamma(k_1+\nu-u)\Gamma(k_2+u-\frac{1}{2})\Gamma(k_2-\frac{3}{2})}\sum_{t\in \mathbb{Z}} (t+\kappa_{1})^{\nu-u}\frac{d_h(n-t-\kappa_{1}+\kappa_{2},\mu)b(t+\kappa_{1})}{(-n+t+\kappa_{1}-\kappa_{2}+\frac{\mu^2}{4m})^{k_1+k_2+2\nu-\frac{3}{2}}}.
        \end{aligned}
    \end{equation*}
Here, $T^{\mathrm{skew},*}_{g,\nu}$ denotes the adjoint map of $T^{\mathrm{skew}}_{g,\nu}$.
    
\end{thm}


\end{document}